\documentclass[reqno]{amsart}
\usepackage{amssymb}
\usepackage{verbatim}
\usepackage{amsfonts}
\usepackage{amsmath}
\usepackage{color}
\usepackage{enumitem}

\newtheorem{theorem}{Theorem}

\newtheorem{definition}[theorem]{Definition}
\newtheorem{example}[theorem]{Example}

\newtheorem{lemma}[theorem]{Lemma}

\newtheorem{proposition}[theorem]{Proposition}
\newtheorem{remark}[theorem]{Remark}

\def\eps{\varepsilon}
\def\ZZ{\mathbb{Z}}
\def\NN{\mathbb{N}}

\def\RR{\mathbb{R}}
\def\PP{\mathbb{P}}
\def\EE{\mathbb{E}}
\def\F{\mathcal{F}}
\def\E{\mathcal{E}}

\def\N{\mathcal{N}}

\DeclareMathOperator{\Poisson}{Poisson}
\DeclareMathOperator{\Leb}{Leb}

\DeclareMathOperator{\Prob}{Prob}
\DeclareMathOperator*{\essinf}{ess\,inf}
\DeclareMathOperator*{\esssup}{ess\,sup}

\makeatletter
\newcommand{\proofstep}[1]{%
  \par% ensure starting on a new paragraph
  \addvspace{\medskipamount}% some vertical space
  \textit{#1\@addpunct{.}}\enspace\ignorespaces
}
\makeatother

\makeatletter
\newcommand*{\intf}{%
  \@ifnextchar_{\intf@sub}{%
    \@ifnextchar^{\intf@sup}{%
      \intf@{}{}%
    }%
  }%
}
\def\intf@sub_#1{%
  \@ifnextchar^{%
    \intf@sub@sup{#1}%
  }{%
    \intf@{#1}{}%
  }%
}
\def\intf@sup^#1{%
  \@ifnextchar_{%
    \intf@sup@sub{#1}%
  }{%
    \intf@{}{#1}%
  }%
}
\def\intf@sub@sup#1^#2{\intf@{#1}{#2}}
\def\intf@sup@sub#1_#2{\intf@{#2}{#1}}
\def\intf@#1#2#3d#4{%
  \int
  \ifx\\#1\\\else _{#1}\fi
  \ifx\\#2\\\else ^{#2}\fi
  \!#3\,\mathrm{d}#4%
}
\makeatother

\newcommand\restr[2]{{% we make the whole thing an ordinary symbol
  \left.\kern-\nulldelimiterspace % automatically resize the bar with \right
  #1 % the function
  \vphantom{\big|} % pretend it's a little taller at normal size
  \right|_{#2} % this is the delimiter
  }}

\newcommand{\abs}[1]{\left\lvert#1\right\rvert}
\newcommand{\norm}[1]{\left\lVert#1\right\rVert}

\begin{document}

\title{Quenched Poisson processes for random subshifts of finite type}
\author{Harry Crimmins}

\address{School of Mathematics and Statistics \\
University of New South Wales, Sydney NSW 2052,
Australia}
\email{harry.crimmins@unsw.edu.au}

\author{Beno{\^i}t Saussol}
\address{Univ Brest, UMR CNRS 6205, Laboratoire de Math\'ematiques de Bretagne Atlantique\\ and Aix Marseille Universit\'e, CNRS, Centrale Marseille, Institut de Math\'ematiques de Marseille, I2M - UMR 7373, 13453 Marseille, France}
\email{benoit.saussol@univ-amu.fr}

\date{\today}

\begin{abstract}
In this paper we study the quenched distributions of hitting times for a class of random dynamical systems.
We prove that hitting times to dynamically defined cylinders converge to a Poisson point process under the law of random equivariant measures with super-polynomial decay of correlations.
In particular, we apply our results to uniformly aperiodic random subshifts of finite type equipped with random invariant Gibbs measures.
We emphasize that we make no assumptions about the mixing property of the marginal measure.
\end{abstract}


\maketitle

\section{Introduction}

Let $(\Omega,\theta,\PP)$ be an invertible ergodic measure preserving transformation, and for each $\omega \in \Omega$ let $ X_\omega \subset X$ be a subset of a fixed compact metric space $X$, and let $f_\omega \colon X_\omega\to X_{\theta\omega}$ be a bi-measurable transformation.
Consider the associated random dynamical system described on the bundle $\E=\bigsqcup_{\omega \in \Omega} \{\omega\} \times X_\omega$ by the skew product $S\colon \E\to\E$ given by $S(\omega,x)=(\theta\omega,f_\omega(x))$.
Let $\nu$ be an $S$-invariant probability measure with marginal $\PP$ on $\Omega$, let $\{\mu_\omega\}_{\omega \in \Omega}$ denote its disintegration into sample measures on the fibers $\{X_\omega\}_{\omega \in \Omega}$, and let $\mu=\int_\Omega\mu_\omega d\PP$ denote the marginal of $\nu$ on $X$.
For each $k \in \ZZ^+$ we denote the random composition $f_{\theta^{k-1}\omega}\circ\cdots\circ f_\omega$ by $f_\omega^{k}$.

Our goal is to prove that the entrance times to a set $A \subset X$ are well approximated by a Poisson distribution in the limit of small $A$. For a deterministic, ergodic system $(X,f,\mu)$ the mean return time to a set $A$ is $\mu(A)^{-1}$ and one must consequently rescale time by a factor of $\mu(A)$ in order to obtain non-degenerate results.
For example it has been proven in many situations that the hitting time point process
\begin{equation*}
\N_A(x) := \sum_{k=1}^\infty 1_A(f^kx) \delta_{k\mu(A)}
\end{equation*}
is well approximated by a uniform Poisson point process, where $\delta_s$ is the Dirac mass at the time $s\ge0$. See \cite{P,Co,AG,C} and references therein for a broad overview, and especially \cite{bookEVT} where the link between hitting time statistics and extreme value theory is discussed.

For some random systems this deterministic scaling is still the proper one. This is the case, for instance, in \cite{RSV}, where it was shown that the first hitting times to cylinder sets about non-periodic points for random subshifts of finite type are approximated by an exponential distribution, provided that one has sufficiently fast mixing of the marginal measure $\mu$.
These results were extended in \cite{rousseau2015hitting,haydn2016return,AFV,FFV} to additionally characterise the distribution of the first and $n$th hitting times to cylinders about periodic points.
The goal of this paper is to obtain such results without mixing assumptions on the base map $\theta$, the $S$-invariant measure $\nu$ or the marginal $\mu$.
The key insight is that it is necessary to keep a \emph{random scaling} for the time. More precisely, let
\begin{equation}\label{eq:time}
T_\omega^k(A) = \sum_{i=1}^k\mu_{\theta^i\omega}(A).
\end{equation}
We observe that the expected value of $T_\omega^k(A)$ is equal to $k\mu(A)$, in accordance with the scaling of the deterministic case.
We then define the hitting time point process
\begin{equation}\label{eq:htpp}
\N_A(\omega,x) = \sum_{k=1}^\infty  1_{A}(f_\omega^{k}x) \delta_{T_\omega^k(A)},
\end{equation}
Under the law of $\mu_\omega$ this is a random point measure, and we will show that in some situations it converges to a uniform Poisson Point Process, $\omega$-a.s. (a so called quenched limit theorem), as $A$ shrinks to a point.

For smooth systems the existence of extreme value laws for sufficiently regular observables is closely related to statistical properties of hitting times \cite{FFT1}.
This connection has been recently explored for random dynamical systems in \cite{FFV}, however this approach does not cleanly adapt to the study of return times to cylinder sets for random subshifts of finite type (for an extended discussion on this topic see e.g. \cite[Section 5]{FFT2}).
We also mention \cite{FFMV} where point processes are used to study the extreme behaviour of random dynamical systems.
However, the approach of \cite{FFMV}, being built upon the setting of \cite{FFV}, suffers the same shortcoming as \cite{FFV} when characterizing the return time statistics of cylinder sets for random subshifts of finite type.
Moreover, the application to random dynamical systems in \cite{FFMV} requires that the marginal measure exhibit polynomial decay of correlations (specifically, in order to prove the limit \cite[(4.6)]{FFMV}).

An outline of the paper is as follows. In Section \ref{sec:prelim} we review the basic theory of point processes, including a simplified version of Kallenberg's Theorem, which is the main tool we will use to prove the convergence of hitting time point processes.
In Section \ref{sec:gal} we examine the behaviour of the hitting time point process for general random dynamical systems with a view towards verifying the hypotheses of Kallenberg's theorem.
Section \ref{sec:rsft} contains our main results: for random subshifts of finite type we develop tractable conditions that guarantee the a.s. convergence of the hitting time point process to a uniform Poisson point process without assuming any mixing properties of $\nu$, $\mu$, or $\theta$ (Theorem \ref{thm:main1}).
We then show that these conditions are satisfied when $\nu$ is a random invariant Gibbs measure (Theorem \ref{thm:gibbs}).
Lastly, in Section \ref{sec:examples} we detail a class of examples satisfying the requirements of Theorem \ref{thm:gibbs}.
Here we consider an example that could not be treated with earlier techniques: we consider a random subshift of finite type for which $(\Omega, \theta, \mathbb{P})$ is ergodic but not mixing, and where $\nu$ is a random invariant Gibbs measure such that the marginal $\mu$ is not mixing.

\section{Preliminaries on point processes}\label{sec:prelim}

In this section we review the basic theory of point processes that we require in later sections.
This material summarises the content of \cite[Chapter 3]{resnick} for the case of (temporal) point process on $[0,\infty)$.
A point measure on $[0, \infty)$ is any Radon measure $m$ on $[0, \infty)$ of the form
\begin{equation*}
	m = \sum_i \delta_{x_i},
\end{equation*}
where $\{x_i\} \subseteq [0, \infty)$ is a countable collection of (not necessarily distinct) points.
We denote the set of point measures on $[0, \infty)$ by $P([0, \infty))$ and endow $P([0, \infty))$ with the smallest sigma-algebra for which that map $m \in P([0, \infty)) \mapsto m(A)$ is measurable for every Borel set $A$.
A point process on $[0, \infty)$ is a $P([0, \infty))$-valued random variable.

We have already introduced a point process of relevance to (random) dynamical systems: the hitting time point processes defined in \eqref{eq:htpp}.
There is some flexibility in the definition here so let us be precise: for each $\omega \in \Omega$ we consider the point process $x \in X_\omega \mapsto \mathcal{N}_A(\omega, x) \in P([0, \infty))$.
Another important class of point processes are the Poisson point processes, which we now define.

\begin{definition}[Poisson point process]
	Suppose $\eta$ is a Radon measure on $[0, \infty)$.
	A point process $m$ on $[0, \infty)$ is called a Poisson point process (on $[0, \infty)$) with mean measure $\eta$ if
	\begin{enumerate}
		\item for every Borel set $A$ with $\eta(A) < \infty$ we have $m(A) \sim \Poisson(\eta(A))$;
		\item for every Borel set $A$ with $\eta(A) = \infty$ we have $m(A) = \infty$ a.s.; and
		\item if $\{A_i\}_{i=1}^k$ are disjoint Borel sets then $\{m(A_i)\}_{i=1}^k$ are independent random variables.
	\end{enumerate}
	Lastly, we call a Poisson point process on $[0,\infty)$ with Lebesgue mean measure a uniform Poisson point process (on $[0,\infty)$).
\end{definition}

Given a sequence of point measures $\{m_k\}_{k \in \ZZ^+} \subseteq P([0, \infty))$ and some $m \in P([0, \infty))$ we say that $\lim_{k \to \infty} m_k = m$ iff $\lim_{k \to \infty} m_k(B) = m(B)$ for all pre-compact $B \subseteq [0, \infty)$ such that $m(\partial B ) = 0$.
This notion of convergence on $P([0, \infty))$ is metrizable (\cite[Proposition 3.17]{resnick}), and makes $P([0, \infty))$ into a complete, separable metric space.
If $(\Omega, \mathcal{A}, \mathbb{P})$ is a probability space and $\{\mathcal{P}_n : \Omega \to P([0,\infty))\}_{n \in \ZZ^+}$ is a sequence of point processes then we say that $\{ \mathcal{P}_n \}_{n \in \ZZ^+}$ converges weakly to the point process $\mathcal{P} : \Omega \to P([0, \infty))$ if for every Borel set $A  \subseteq P([0, \infty))$ such that $\Prob_{\mathbb{P}}(\mathcal{P} \in \partial A) = 0$ one has
\begin{equation*}
	\lim_{n \to \infty} \Prob_{\mathbb{P}}(\mathcal{P}_n \in A) = \Prob_{\mathbb{P}}(\mathcal{P} \in A).
\end{equation*}
In practice, however, this is a rather difficult criteria to check, and more concrete conditions are supplied by Kallenberg's Theorem (\cite[Proposition 3.22]{resnick}).
We now state a simplified version of Kallenberg's Theorem that is adapted to our purposes; namely, for proving the convergence of a sequence of point processes to a uniform Poisson point process.
We denote Lebesgue measure on $[0, \infty)$ by $\Leb$.

\begin{theorem}\label{thm:kallenberg}
	Suppose that $(\Omega, \mathcal{A}, \mathbb{P})$ is an atomless standard probability space, that $\mathcal{P}_n : \Omega \to P([0,\infty))$ is a point process for each $n \in \ZZ^+$, and that for any finite union $R=\cup_i R_i$ of disjoint open bounded intervals in $[0,\infty)$ we have
 	\begin{enumerate}[label=(K\arabic*)]
 		\item \label{k1} $\EE_{\mathbb{P}}\left(\mathcal{P}_{n}(R)\right)\to \Leb(R)$; and
 		\item \label{k2} $\Prob_{\mathbb{P}}\left(\mathcal{P}_{n}(R)=0\right) \to e^{-\Leb(R)}$.
	\end{enumerate}
	Then $\mathcal{P}_n$ converges weakly to a uniform Poisson point process $\mathcal{P} : \Omega \to P([0, \infty))$.
\end{theorem}

Unlike Theorem \ref{thm:kallenberg}, Kallenberg's Therorem does not guarantee the existence of a Poisson point process on $(\Omega, \mathcal{A}, \mathbb{P})$.
To deduce Theorem \ref{thm:kallenberg} from Kallenberg's Theorem (\cite[Proposition 3.22]{resnick}) it suffices to show that a uniform Poisson point process $\mathcal{P} : \Omega \to P([0,\infty))$ exists, which follows from the assumption that $(\Omega, \mathcal{A},\mathbb{P})$ is an atomless standard probability space\footnote{We will briefly outline the construction.
	Since $(\Omega, \mathcal{A},\mathbb{P})$ is an atomless standard probability space it is almost isomorphic to the Hilbert cube $[0,1]^{\NN}$ equipped with the infinite product of Lebesgue measures $\Leb^{\NN}$ (as produced by e.g. \cite[Theorem 7.28]{folland2013real}). The identity on $[0,1]^{\NN}$ defines a i.i.d. sequence of uniform random variables $\{U_i\}_{i \in \NN}$. It is routine to verify that
	\begin{equation*}
	\sum_{i=1}^\infty \delta_{\ln U_1 + \cdots + \ln U_i}
	\end{equation*}
	is a uniform Poisson point process with domain $[0,1]^{\NN}$.
	Since $(\Omega, \mathbb{P})$ and $([0,1]^{\NN}, \Leb^{\NN})$ are almost isomorphic it follows that there exists a uniform Poisson point process on $(\Omega, \mathcal{A},\mathbb{P})$.
}.

%%%%%%%%%%%%%%%%%%%%%%%%%%%
\section{Estimates for general random systems and sets}\label{sec:gal}

In this section we describe some general results on the verification of \ref{k1} and \ref{k2} for random dynamical systems. Our approach is loosely based on \cite{RSV} (see also \cite{rousseau2015hitting}).
Let $(\Omega, \mathcal{A}, \mathbb{P})$ be an atomless standard probability space, $\theta : \Omega \to \Omega$ be a $\mathbb{P}$-ergodic map with measurable inverse, $(Y,d)$ be a compact metric space and $\mathcal{B}$ be the Borel sigma-algebra on $Y$.
With $\pi_\Omega : \Omega \times Y \to \Omega$ denoting the projection onto the first coordinate, we call $E \in \mathcal{A} \times \mathcal{B}$ a compact, measurable bundle over $\Omega$ if $\pi_\Omega(E) = \Omega$ and $\pi_\Omega^{-1}(\omega) := Y_\omega$ is compact for every $\omega \in \Omega$.
\begin{definition}
  If $E$ is a compact, measurable bundle over $\Omega$ then we say that $S : E \to E$ is a random dynamical system covering $\theta$ if
  \begin{enumerate}
    \item $\pi_\Omega \circ S = \theta \circ \pi_\Omega$ i.e. $S(Y_\omega) \subseteq Y_{\theta \omega}$ for every $\omega \in \Omega$; and
    \item $S$ is measurable with respect to the relative product sigma-algebra on $E$.
  \end{enumerate}
\end{definition}
For the remainder of this section we fix a compact, measurable bundle $E$ over $\Omega$ and a random dynamical system $S : E \to E$ covering $\theta$.
For each $\omega \in \Omega$ we denote the map $x \in \pi_Y(Y_\omega) \mapsto  \pi_Y(S(\omega, x)) \in \pi_Y(Y_{\theta\omega})$ by $f_\omega$ and set $f_\omega^n := f_{\theta^{n-1} \omega} \circ \cdots \circ f_\omega$ for every $n \in \ZZ^+$.

Suppose $\nu$ is a $S$-invariant probability measure on $\Omega \times Y$ such that $\nu \circ \pi_\Omega^{-1} = \mathbb{P}$.
Then there exists a disintegration of $\nu$ with respect to $\mathbb{P}$ into \emph{sample measures} $\{ \mu_\omega \}_{\omega \in \Omega}$ on $Y$ i.e. for every $h \in L^1(\nu)$ one has
\begin{equation*}
  \intf_{\Omega \times Y} h(\omega, x) d\nu = \intf_\Omega \intf_Y h(\omega,x ) d\mu_\omega d\mathbb{P}.
\end{equation*}
Moreover this disintegration is unique $\mathbb{P}$-a.s.
Due to the invertibility of $\theta$, the $S$-invariance of $\nu$, and the $\mathbb{P}$-a.s. uniqueness of the disintegration one deduces that the set
\begin{equation*}
  \Omega' = \{\omega\in\Omega\colon \forall i,\ (f_\omega^i)_*\mu_\omega=\mu_{\theta^i\omega}\}
\end{equation*}
has full $\mathbb{P}$-measure.

Fix a measurable set $A \subseteq Y$.
We aim to examine the conditions of Theorem \ref{thm:kallenberg} for the hitting time point process $ \mathcal{N}_A$ (as defined in \eqref{eq:htpp}) in the limit as $\eps(A) := \esssup_\omega\mu_\omega(A) \to 0$.
To this end, suppose that $R= \cup_{i=1}^r R_i$ is a finite union  of disjoint open bounded intervals in $[0, \infty)$ and define
\begin{equation}\begin{split}\label{eq:constants}
  p_i &= \max \{k\ge1\colon T_\omega^k(A)\le \inf R_i\}, \\
  q_i &= \max\{k\ge1\colon T_\omega^k(A)\le \sup R_i\}-p_i,\\
  k_* &= \max_i p_i+q_i.
\end{split}\end{equation}
While $p_i, q_i$ and $k_*$ depend on $\omega$, $A$ and $R$, for simplicity we suppress these dependencies from our notation.
Notice that by the ergodic theorem all the constants defined in \eqref{eq:constants} are a.e. finite provided that $\eps(A) \ne 0$.
From \eqref{eq:constants} and the definition of $\mathcal{N}_A$ we see that
\begin{equation}\label{eq:htpp_2}
  \N_A(\omega,x)(R) = \sum_{j=1}^{k_*} 1_A(f_\omega^j x) 1_R(T_\omega^j(A)) = \sum_{i=1}^r \sum_{j=p_i+1}^{p_i + q_i}1_A(f_\omega^jx).
\end{equation}

Our first result for this section confirms that \ref{k1} is satisfied as $\eps(A)$ vanishes.
\begin{lemma}\label{lem:K1_proof}
	The expected number of points in $R$ is close to $\Leb(R)$: if $\eps(A) \ne 0$ then for $\mathbb{P}$-a.e. $\omega$ we have
	\begin{equation*}
	   \abs{\EE_{\mu_\omega}\left(\N_A(\omega,\cdot)(R)\right) - \Leb(R)} \le 2r \eps(A).
	\end{equation*}
	\begin{proof}
		By taking the expectation with respect to $\mu_\omega$ of \eqref{eq:htpp_2} and then using the $\mathbb{P}$-a.e. equivariance of $\{\mu_{\omega}\}_{\omega \in \Omega}$ under $S$ we have for $\mathbb{P}$-a.e. $\omega$ that
		\begin{equation}\begin{split}\label{eq:K1_proof_1}
			\EE_{\mu_\omega}\left(\N_A(\omega,\cdot)(R)\right) = \sum_{i=1}^r \sum_{j=p_i+1}^{p_i + q_i}\mu_{\omega}((f_\omega^j)^{-1}(A))
			&= \sum_{i=1}^r \sum_{j=p_i+1}^{p_i + q_i}\mu_{\theta^j \omega}(A) \\
			&= \sum_{i=1}^r T^{p_i+q_i}_\omega(A) - T^{p_i}_\omega(A).
		\end{split}\end{equation}
		By the definition of $p_i$ and $q_i$ we have
		\begin{equation*}
			T^{p_i}_\omega(A) \le \inf R_i < T^{p_i+1}_\omega(A)
		\end{equation*}
		and
		\begin{equation*}
			T^{p_i+q_i}_\omega(A) \le \sup R_i < T^{p_i+q_i+1}_\omega(A).
		\end{equation*}
		Hence $\abs{T^{p_i}_\omega(A) - \inf R_i} \le \mu_{\theta^{p_i + 1}\omega}(A)$ and $\abs{T^{p_i+q_i}_\omega(A) - \sup R_i} \le \mu_{\theta^{p_i +q_i + 1}\omega}(A)$.
		Applying these inequalities to \eqref{eq:K1_proof_1} and recalling the definition of $\eps(A)$ yields the required claim.
	\end{proof}
\end{lemma}

For the remainder of this section we will develop a criteria for verifying \ref{k2} as $\eps(A) \to 0$.
For $I \subseteq \mathbb{Z}^+$ we set
\begin{equation}\label{eq:f_sets}
	F_\omega^A(I) = \bigcap_{j \in I} (f_\omega^j)^{-1}(A^c),
\end{equation}
and for $k \in \ZZ^+$ we define
\begin{equation*}
	\Delta_\omega(A,k) = \sup_{I \subseteq \mathbb{Z}^+ \cap [0, k]} \abs{\mu_\omega(A\cap F_\omega^A(I)) - \mu_\omega(A)\mu_\omega(F_\omega^A(I))}.
\end{equation*}

\begin{lemma}\label{lem:delta}
	If $\eps(A) \ne 0$ then for all $\omega\in\Omega'$ we have % that
	%for any $r\ge1$, any non-negative integers $p_1,\ldots,p_r$, and positive integers $q_1,\ldots,q_r$ such that $p_i+q_i<p_{i+1}$ for any $i=1,\ldots,r-1$ we have
	\begin{equation*}
	   \abs{\Prob_{\mu_\omega}(\mathcal{N}_A(\omega, \cdot)(R) = 0) -\prod_{i=1}^r\prod_{j=p_i+1}^{p_i+q_i}\left(1-\mu_{\theta^j\omega}(A)\right)}
	    \le
	    \sum_{i=1}^{r}\sum_{j=p_i+1}^{p_i+q_i}\Delta_{\theta^j\omega}(A,k_*-j).
  \end{equation*}
	\begin{proof}
    %\abs{\mu_\omega(\forall i, \tau_A^{\theta^{p_i}\omega}\circ f_\omega^{p_i}>q_i)-\prod_{i=1}^r\prod_{j=p_i+1}^{p_i+q_i}\left(1-\mu_{\theta^j\omega}(A)\right)}\le\sum_{i=1}^{r}\sum_{j=p_i+1}^{p_i+q_i}\Delta_{\theta^j\omega}(A,k_*-j).
		The details of the case where $r=1$ are subsumed by that of the case where $r \ge 2$, so we will only deal with the later case.
    For $a,b \in \NN$ with $b \ge 1$ set
		\begin{equation*}
		    U_{\omega}(a,b) %= \{ x : (\tau_A^{\theta^{a}\omega}\circ f_\omega^{a})(x)>b \}
        = \bigcap_{j=a+1}^{a+b} (f_\omega^j)^{-1}(A^c).
		\end{equation*}
    By the definition of $p_i$ and $q_i$ we have
    \begin{equation*}
      \mathcal{N}_A(\omega, x)(R) = 0 \iff x \in \bigcap_{i=1}^r U_{\omega}(p_i, q_i).
      %\mu_\omega(\forall i, \tau_A^{\theta^{p_i}\omega}\circ f_\omega^{p_i}>q_i).
    \end{equation*}
		Since $\omega \in \Omega'$ we have
		\begin{equation*}
			\mu_\omega\left(\bigcap_{i=1}^{r} U_{\omega}(p_i, q_i)\right) = \mu_{\theta^{p_1}\omega}\left(\bigcap_{i=1}^{r} U_{\theta^{p_1}\omega}(p_i - p_1, q_i)\right).
		\end{equation*}
    We set $I = \mathbb{Z} \cap [1, q_1 - 1] \cap \left(\bigcap_{j=2}^r [p_j - p_1, p_j + q_j -p_1 - 1]\right)$ so that
		\begin{equation*}
			U_{\theta^{p_1+1}\omega} (0, q_1 - 1)\cap \left(\bigcap_{i=2}^{r}U_{\theta^{p_1+1}\omega}(p_i - p_1 - 1, q_i)\right) = F_{\theta^{p_1 +1}\omega}^A(I).
		\end{equation*}
		It follows that
		\begin{equation}\begin{split}\label{eq:delta_2}
			\mu_\omega\left(\bigcap_{i=1}^{r} U_{\omega}(p_i, q_i)\right)= &\mu_{\theta^{p_1}\omega}\left(f_{\theta^{p_1+1}\omega}^{-1}\left(A^c \cap F_{\theta^{p_1 +1}\omega}^A(I) \right)\right) \\
			= &\mu_{\theta^{p_1+1}\omega}
			\left( F_{\theta^{p_1 +1}\omega}^A(I) \right) -\mu_{\theta^{p_1+1}\omega}
			\left( A \cap F_{\theta^{p_1 +1}\omega}^A(I) \right).
		\end{split}\end{equation}
		As $I \subseteq [1, k_* - p_1 - 1]$, by applying the definition of $\Delta_{\theta^{p_1 +1} \omega}(A, k_* - p_1 -  1)$ we find that
		\begin{equation*}
			\abs{\mu_\omega\left(\bigcap_{i=1}^{r} U_{\omega}(p_i, q_i)\right) - (1 - \mu_{\theta^{p_1 +1}\omega}(A)) \mu_{\theta^{p_1 +1}\omega}\left( F_{\theta^{p_1 +1}\omega}^A(I) \right) } \le \Delta_{\theta^{p_1 +1} \omega}(A, k_* - p_1 - 1 ).
		\end{equation*}
		By iterating the argument from \eqref{eq:delta_2} until now we obtain the bound
		\begin{equation*}\begin{split}
			&\abs{\mu_\omega\left(\bigcap_{i=1}^{r} U_{\omega}(p_i, q_i)\right) - \left(\prod_{j=p_1+1}^{p_1 + q_1}(1 - \mu_{\theta^{j}\omega}(A)) \right)\mu_{\theta^{p_1 +q_1}\omega}\left(  \bigcap_{i=2}^{r}U_{\theta^{p_1+q_1}\omega}(p_i - p_1 - q_1, q_i)\right) } \\
			&\le \sum_{j=p_1+1}^{p_1+q_1} \Delta_{\theta^{j} \omega}(A, k_* - j).
		\end{split}\end{equation*}
		The claim is obtained by recursively applying the entire argument thus far to estimate
		\begin{equation*}
			\mu_{\theta^{p_1 +q_1}\omega}\left(  \bigcap_{i=2}^{r}U_{\theta^{p_1+q_1}\omega}(p_i - p_1 - q_1, q_i)\right) .
		\end{equation*}
	\end{proof}
\end{lemma}

\begin{lemma}\label{lem:exp}
	For $\mathbb{P}$-a.e. $\omega$ we have
	\begin{equation*}
    \lim_{\eps(A)\to 0} \prod_{i=1}^r\prod_{j=p_i+1}^{p_i+q_i}\left(1-\mu_{\theta^j\omega}(A)\right) = e^{-\Leb(R)}.
  \end{equation*}
	\begin{proof}
		Recall from the proof of Lemma \ref{lem:K1_proof} that
		\begin{equation*}
			\Leb(R_i) = \sup R_i - \inf R_i = \sum_{j=p_i +1}^{p_i + q_i} \mu_{\theta^j\omega}(A) + O(\epsilon(A)).
		\end{equation*}
		The lemma is then a consequence of the following simple and instrumental inequality:
		if $0<\eps\le 1/2$ and $x_1,\ldots,x_k \in [0,\eps]$ then
		\[
		\exp\left(-(1+2\eps)\sum_{i=1}^k x_i\right) \le \prod_{i=1}^k(1-x_i) \le \exp\left(-(1-2\eps)\sum_{i=1}^k x_i\right).
		\]
	\end{proof}
\end{lemma}

By Lemmas \ref{lem:delta} and \ref{lem:exp} we may conclude that \ref{k2} holds for $\mathcal{N}_A(\omega, \cdot)$ as $\eps(A) \to 0$ provided that
\begin{equation}\label{eq:decay}
  \lim_{\epsilon(A) \to 0} \sum_{i=1}^{k_*}\Delta_{\theta^i\omega}(A,k_*-i) = 0.
\end{equation}
Our strategy for obtaining this limit is to decompose $\sum_{i=1}^{k_*}\Delta_{\theta^i\omega}(A,k_*-i)$ into a mixing term, a short entrance time term, and a return time term as follows.
We define the random hitting time of a point $x$ to $A$ to be
\begin{equation*}
	\tau^\omega_A(x) := \inf\{ k > 0 : f_\omega^k(x) \in A \}.
\end{equation*}
For $g,k \in \ZZ^+$ set
\begin{equation}\begin{split}\label{eq:terms}
		\Delta_\omega(A,k,g) &= \begin{cases}
		\sup_{I \subseteq \mathbb{Z}^+ \cap (g, k]}  \abs{\mu_\omega(A\cap F_\omega^A(I)) - \mu_\omega(A)\mu_\omega(F_\omega^A(I))} & g < k, \\
		0 & \text{otherwise},
		\end{cases}\\
		G_{A,k,g}(\omega) &= \sum_{i=1}^{k} \mu_{\theta^i\omega}(A\cap \{\tau_A^{\theta^i\omega}\le g\}),\\
		H_{A,k,g}(\omega) &= \sum_{i=1}^{k} \Delta_{\theta^i\omega}(A,k-i,g) ,\\
		K_{A,k,g}(\omega) &= \sum_{i=1}^{k} \mu_{\theta^i\omega}(A) \, \mu_{\theta^i\omega}(\tau_A^{\theta^i\omega}\le g).
	\end{split}
\end{equation}
The quantities in \eqref{eq:terms} may be interpreted as follows: $H_{A,k,g}$ quantifies the amount of mixing exhibited by the system after $g$ units of time, while $G_{A,k,g}$ and $K_{A,k,g}$ quantify the probabilities of returning to and hitting $A$, respectively, before time $g$.
These quantities may be used to prove \eqref{eq:decay} due to the following lemma.

\begin{lemma}\label{lem:somme}
	For all $\omega\in\Omega'$ and any integers $g\le k$ we have
  \begin{equation*}
     \sum_{i=1}^k\Delta_{\theta^i\omega}(A,k-i) \le G_{A,k,g}(\omega)+H_{A,k,g}(\omega)+K_{A,k,g}(\omega).
  \end{equation*}
	\begin{proof}
    If the set $I \subseteq \ZZ^+ \cap [1,k-i]$ realising the supremum in the definition of $\Delta_{\theta^i\omega}(A,k - i)$ is a subset of $[g+1, k-i]$ then we have $\Delta_{\theta^i\omega}(A,k - i) = \Delta_{\theta^i\omega}(A,k - i, g)$.
    Hence in this case we trivially have
    \begin{equation}\begin{split}\label{eq:somme_2}
      \Delta_{\theta^i\omega}(A,k - i) \le &\Delta_{\theta^i\omega}(A,k - i, g) + \mu_{\theta^i\omega}(A \cap \{\tau^{\theta^i\omega}_{A} \le g\}) \\
      &+ \mu_{\theta^i\omega}(A) \mu_{\theta^i\omega}(\tau^{\theta^i\omega}_{A} \le g).
    \end{split}\end{equation}
    Otherwise, if the supremum is realised by some $I \subseteq \ZZ^+ \cap [1,k-i]$ such that $I \cap [0,g]$ is non-empty then we have the bound
		\begin{equation}\begin{split}\label{eq:somme_1}
			\big|\mu_{{\theta^i\omega}}&(A \cap F_{\theta^i\omega}^A(I)) - \mu_{\theta^i\omega}(A)\mu_{\theta^i\omega}(F_{\theta^i\omega}^A(I))\big| \\
			\le &\abs{\mu_{\theta^i\omega}(A \cap F_{\theta^i\omega}^A(I \cap (g, k])) - \mu_{\theta^i\omega}(A)\mu_{\theta^i\omega}(F_{\theta^i\omega}^A(I \cap (g, k]))} \\
			&+ \mu_{\theta^i\omega}(A) \mu_{\theta^i\omega}(F_{\theta^i\omega}^A(I \cap (g,k]) \cap F_{\theta^i\omega}^A(I \cap [0,g])^c) \\
			&+ \mu_{\theta^i\omega}(A \cap F_{\theta^i\omega}^A(I \cap [g,k]) \cap F_{\theta^i\omega}^A(I \cap [0,g])^c).
		\end{split}\end{equation}
		Since $I \cap [0,g]$ is non-empty we have
		\begin{equation*}
			F_{\theta^i\omega}^A(I \cap (g,k]) \cap F_{\theta^i\omega}^A(I \cap [0,g])^c \subseteq \{ x : \tau^{\theta^i\omega}_{A}(x) \le g\}.
		\end{equation*}
		It follows that
		\begin{equation*}
			\mu_{\theta^i\omega}(A) \mu_{\theta^i\omega}(F_{\theta^i\omega}^A(I \cap (g,k]) \cap F_{\theta^i\omega}^A(I \cap [0,g])^c) \le \mu_{\theta^i\omega}(A) \mu_{\theta^i\omega}(\tau^{\theta^i\omega}_{A} \le g)
		\end{equation*}
		and
		\begin{equation*}
		 \mu_{\theta^i\omega}(A \cap F_{\theta^i\omega}^A( I \cap [g,k]) \cap F_{\theta^i\omega}^A(I \cap [0,g])^c) \le \mu_{\theta^i\omega}(A \cap \{\tau^{\theta^i\omega}_{A} \le g\}).
		\end{equation*}
		Hence by taking supremums in \eqref{eq:somme_1} we obtain \eqref{eq:somme_2} for the case where $I \cap [0,g]$ is non-empty.
		The lemma follows by summing \eqref{eq:somme_2} over $i \in \{ 1, \dots, k\}$.
	\end{proof}
\end{lemma}

To summarize, we have shown that \ref{k2} holds as $\eps(A) \to 0$ provided that there exists some $g_A \in \ZZ^+$ such that $G_{A,k_*,g_A}$, $H_{A,k_*,g_A}$ and $K_{A,k_*,g_A}$ vanish as $\eps(A) \to 0$.
Bounding $G_{A,k_*,g_A}$ and $H_{A,k_*,g_A}$ requires knowledge about the RDS in question. Luckily, $K_{A,k_*,g_A}$ can be controlled with little effort.

\begin{lemma}\label{lem:KisOK}
	For $\omega \in \Omega'$ we have $K_{A,k_*,g}(\omega) \le g \eps(A) (\sup R)$.
  \begin{proof}
		Since
		\begin{equation*}
      \mu_{\omega}(\tau_A^{\omega} \le g) = \mu_\omega\left(\bigcup_{j=1}^g(f_\omega^j)^{-1}A\right)\le \sum_{j=1}^g \mu_{\theta^j\omega}(A) \le \eps(A) g
    \end{equation*}
    and as $\sum_{i=1}^{k_*} = \mu_{\theta^j\omega}(A) = T^{k_*}_\omega(A) \le \sup R$ we have $K_{A,k_*,g}(\omega) \le g (\sup R) \eps(A)$, as required.
	\end{proof}
\end{lemma}

\section{Random subshift of finite type}\label{sec:rsft}

We first give the definition of a random subshift of finite type from \cite{boggun}.
Let $\bar{\ZZ}^+ := \ZZ^+ \cup \{\infty\}$ denote the one-point compactification of $\ZZ^+$ and recall that $X :=(\bar{\ZZ}^+)^\NN$ is a compact metric space when equipped with the metric
\begin{equation*}
  d(x,y) = \sum_{i=0}^\infty 2^{-1} \abs{x_i^{-1} - y_{i}^{-1}},
\end{equation*}
where one takes $1/\infty = 0$. Denote the left-shift on $X$ by $\sigma$.
Let $(\Omega,\mathcal{A},\PP)$ be an atomless standard probability space and fix an invertible, $\mathbb{P}$-ergodic map $\theta : \Omega \to \Omega$ with measurable inverse.
Let $b:\Omega \to \mathbb \ZZ^+$ be a random variable such that  $\ln b \in L^1(\mathbb{P})$.
Let $Q=\left\{Q(\omega)=(a_{ij}(\omega)):\omega\in \Omega\right\}$ be a random transition matrix over $\theta$, i.e. for any $\omega\in\Omega$, $Q(\omega)$ is a $b(\omega)\times b(\theta\omega)$-matrix with entries in $\{0,1\}$, at least one non-zero entry in each row and column, and such that for any $i, j\in\ZZ^+$ the map
\begin{equation*}
  \omega \mapsto
  \begin{cases}
    a_{ij}(\omega) &  j \le b(\omega), i \le b(\theta \omega)\\
    0 & \text{otherwise} \\
  \end{cases}
\end{equation*}
is measurable.
For $\omega\in \Omega$ let $X_\omega = \{1, \dots, b(\omega)\}$ and set
\begin{equation*}
  \E_\omega =\{x=(x_0,x_1,\ldots)\colon x_i \in X_{\theta^i \omega} \text{ and } a_{x_i x_{i+1}}(\theta^i\omega)=1\text{ for all } i\in\NN\}\subset X.
\end{equation*}
Then the set
\begin{equation*}
  \E_Q = \{(\omega,x)\colon \omega\in\Omega,x\in\E_\omega\}
\end{equation*}
is a compact, measurable bundle over $\Omega$, and the map $S_Q : \E_Q \to \E_Q$ that is defined by $S_Q(\omega,x)= (\theta \omega,\sigma x)$ is a random dynamical system covering $\theta$.
A random subshift of finite type is any map arising from this construction (i.e. the construction of $S_Q$).

For the remainder of this section we fix a random subshift of finite type $S_Q$ and assume the existence of a $S_Q$-invariant probability measure $\nu$ with marginal $\PP$ on $\Omega$. We denote by $\{\mu_\omega\}_{\omega \in \Omega}$ the disintegration of $\nu$ into sample measures on $\E_\omega$.
For $y \in X$ we denote by $C_n(y)=\{z \in X : y_i=z_i \text{ for all } 0\le i\le n-1\} $ the \emph{$n$-cylinder} that contains $y$, and let $\F_0^n$ denote the sigma-algebra generated by all $n$-cylinders. We also define $C_0(y) = X$.
For this section we fix a sequence $A_n\in\F_0^n$ and study the (random) hitting time point processes for $A_n$ in the limit as $n \to \infty$.

Our first main result for this section provides conditions under which this process converges a.s. to a uniform Poisson point process.
We assume the following: there is a positive random variable $D \in L^p(\PP)$ for some $p\in(0,1]$, real-valued functions $\beta_0,\beta_1$ and either $\alpha$ or $\phi$ defined on $\NN$ such that for all $m,n$, $A\in\F_0^n$ and $B\in \F_0^{m}$, $\PP$-almost surely $\omega \in \Omega$:
\begin{enumerate}[label=(\Roman*)]
\item \label{sft1} (Cylinder returns) For all $0\le j\le n\le k$
\[
\begin{split}
\mu_\omega(A_n\cap\sigma^{-j}(A_n) &\le \mu_\omega(A_n) \beta_0(j), \text{ and} \\
\mu_\omega(A_n\cap\sigma^{-k}(A_n) &\le \mu_\omega(A_n) \beta_1(n).
\end{split}
\]
\item\label{sft2} (Decay of cylinders) With $\displaystyle  \eps(n) := \esssup_{\omega} \mu_\omega(A_n)$ we have $\displaystyle \lim_{n \to \infty} \eps(n) = 0$ and $\eps(n) \ne 0$ for all $n \in \ZZ^+$.
\item \label{sft3} (Mixing) Either the system is uniformly fibered mixing with rate $\alpha$:
\[
\forall g \in \ZZ^+ \quad \left|\mu_\omega(A\cap\sigma^{-g-n}B) -\mu_\omega(A)\mu_{\theta^{n+g}\omega}(B)\right|\le D(\omega)\alpha(g),
\]
or it is fibered $\phi$-mixing:
\[
\forall g \in \ZZ^+ \quad \left|\mu_\omega(A\cap\sigma^{-g-n}B) -\mu_\omega(A)\mu_{\theta^{n+g}\omega}(B)\right|\le \mu_\omega(A)\phi(g).
\]
\item \label{sft4} Setting $q_n=\min\{j\ge1\colon A_n\cap \sigma^{-j}A_n\neq\emptyset\}$, we assume that
\begin{equation}\label{eq:zero}
\lim_{n\to\infty}\inf_{h\ge 0}
h \beta_1(n) + \left(\sum_{j=q_n}^n \beta_0(j)\right) + \min(\alpha(h) k_n^{\frac1p},\phi(h))+(h+n) \eps(n)
=0.
\end{equation}
\end{enumerate}
The last condition penalises $A_n$ for having short return times relative to $n$. In particular, in the case where $A_n = C_n(y)$ the condition \ref{sft4} discards periodic $y$ (see e.g. \cite[Lemma 4.10]{rousseau2015hitting}).

\begin{theorem}\label{thm:main1}
  If \ref{sft1}-\ref{sft4} hold then the hitting time process
  $$
  \N_{A_n}(\omega,x) := \sum_{k=1}^\infty  1_{A_n}(\sigma^k(x)) \delta_{T_\omega^k(A_n)}.
  $$
  converges for $\PP$ a.e. $\omega$ under the law of $\mu_\omega$ to a uniform Poisson Point Process on $[0,\infty)$.
  In particular, we have $\PP$-a.s. convergence in distribution of the first hitting time:
  \begin{equation}\label{eq:cvinlaw}
	   \lim_{n\to\infty}\sup_{k\ge0}\abs{\mu_\omega\left(x\in X\colon \tau_{A_n}(x)>k\right)- \exp\left(-\sum_{i=1}^k\mu_{\theta^i\omega} (A_n)\right)}=0.
   \end{equation}
\end{theorem}

Before proving Theorem \ref{thm:main1} we describe a specific situation in which \ref{sft1}-\ref{sft4} hold;
namely, when $\nu$ is a $S_Q$-invariant Gibbs measure for a H\"older continuous potential and $A_n=C_n(y)$ for certain non-periodic $y$.
In order to properly state our results we recall the setting of \cite{boggun}.

Let $L^0(\Omega, \mathcal{C}(\E_Q))$ denote the class of families $\psi = \{\psi(\omega) \in \mathcal{C}(\E_\omega, \mathbb{R})\}_{\omega \in \Omega}$ of real-valued, continuous maps such that $(\omega, x) \mapsto \psi(\omega, x)$ is measurable on $\E_Q$, and set
\begin{equation*}
  L^1_{\mathbb{P}}(\Omega, \mathcal{C}(\E_Q)) = \left\{ \psi \in L^0(\Omega, \mathcal{C}(\E_Q)) : \intf \norm{\psi(\omega, \cdot)}_{\mathcal{C}(\E_\omega,\mathbb{R})} d \mathbb{P} < \infty \right\}.
\end{equation*}
Let $\mathbb{F}_Q$ denote set of $\psi \in L^0(\Omega, \mathcal{C}(\E_Q))$ for which there exists $a \ge 0$ and $r \in (0,1)$ such that for every $n \in \NN$ and $\mathbb{P}$-a.e. $\omega$ one has
\begin{equation}\label{eq:holder_obs}
  \sup \left\{ \abs{ \psi(\omega,x) - \psi(\omega,y)}\, : \,x,y \in \E_\omega, x \in C_n(y) \right\} \le a r^n.
\end{equation}

We now give the definition of a random Gibbs measure from \cite{boggun}.
\begin{definition}\label{def:gibbs}
  Suppose that $S_Q$ is a random subshift of finite type and that $\nu$ is measure on $\E_Q$ with marginal $\mathbb{P}$ on $\Omega$ and disintegration $\{\mu_{\omega}\}_{\omega \in \Omega}$.
  We say that $\nu$ is a random Gibbs measure for the potential $\psi \in \mathbb{F}_Q$ if there exists $H \in L^1_{\mathbb{P}}(\Omega,\mathcal{C}(\E_Q))$ such that $\ln H \in \mathbb{F}_Q$ and so that if we define
  \begin{equation*}
    F_n(\omega, x) = \frac{\exp\left(\sum_{i=0}^{n-1} \psi(\theta^i\omega, \sigma^ix)\right) H(\omega,x)}{\sum_{s \in \E_\omega, \, \sigma^n x = \sigma^n s}\exp\left(\sum_{i=0}^{n-1} \psi(\theta^i\omega, \sigma^i s)\right) H(\omega,s)}
  \end{equation*}
  then for every $n \in \NN$ and $f \in L^1_{\mathbb{P}}(\Omega, \mathcal{C}(\E_A))$ one has
  \begin{equation*}
    \intf f(\omega, x) d\mu_{\omega}(x) = \intf \sum_{z \in \sigma^{-n}(s)\cap \E_\omega}f(\omega, z) F_n(\omega, z) d\mu_{\theta^n \omega}(s) \quad \mathbb{P}\text{-a.s.}
  \end{equation*}
\end{definition}

We formulate two additional conditions\footnote{These are (C1) and (C2) from \cite{boggun}.}, the first is a uniform aperiodicity condition on the product of the random transition matrix $Q$:
\begin{enumerate}[label=(G1)]
  \item \label{g1} There exists $M > 0$ such that $Q(\omega) \cdots Q(\theta^{M-1}\omega)$ has no zero entries for all $\omega \in \Omega$.
\end{enumerate}
The second concerns the boundedness of the transfer operator associated to a potential $\psi \in \mathbb{F}_Q$, although we suppress operator-theoretic details in our formulation:
\begin{enumerate}[label=(G2)]
  \item \label{g2} We have
  \begin{equation*}
    \intf \sup_{x \in \E_\omega} \abs{\ln\left(\sum_{y \in \E_{\theta^{-1}(\omega)} \cap \sigma^{-1}(x)} \exp(\psi(\theta^{-1}\omega, y))\right)} d\mathbb{P}(\omega) < \infty.
  \end{equation*}
\end{enumerate}
We recall from \cite{boggun} the following existence and uniqueness result for random $S_Q$-invariant Gibbs measures.
\begin{theorem}[{\cite[Corollary 4.12]{boggun}}]\label{thm:gibbs_existence}
  If $b \ge 2$ a.s., \ref{g1} holds, and $\psi \in \mathbb{F}_Q$ satisfies \ref{g2} there there is a unique $S_Q$-invariant random Gibbs probability measure for the potential $\psi$.
\end{theorem}
Our second main result for this section is the following.

\begin{theorem}\label{thm:gibbs}
  Suppose that $b \ge 2$ a.s., that \ref{g1} holds, and that $\psi \in \mathbb{F}_Q$ satisfies \ref{g2}. Let $\nu$ be the unique $S_Q$-invariant random Gibbs probability measure for $\psi$.
  If $y \in X$ is non-periodic and $\nu(C_n(y)) > 0$ for every $n \in \ZZ^+$ then \ref{sft1}-\ref{sft4} hold:
  \begin{enumerate}[label=(\Roman*)]
    \item There exists $c,a > 0$ such that if $\beta_0(n) = c\exp(-an) = \beta_1(n)$ then for all $0 \le j \le n \le k$ and a.e. $\omega$ we have
    \begin{equation*}\begin{split}
      \mu_\omega(C_{j}(y)\cap\sigma^{-j}(C_{n}(y)) &\le \mu_\omega(C_{n}(y)) \beta_0(j), \text{ and} \\
      \mu_\omega(C_{n}(y)\cap\sigma^{-k}(C_{n}(y)) &\le \mu_\omega(C_{n}(y)) \beta_1(n).
    \end{split}\end{equation*}
    \item The measure of cylinders decays exponentially fast and is non-degenerate i.e. there exists $c,a>0$ such that $0 < \eps(n) \le ce^{-an}$ for every $n \in \ZZ^+$.
    \item The system is a.s. fibered $\phi$-mixing with $\phi(n) = O(\exp(-an))$ for some $a > 0$.
    \item We have $q_n \to \infty$, and upon setting $h_n = n$ one has as a consequence of \ref{sft1}-\ref{sft3} that
    \begin{equation*}
      \lim_{n\to\infty} h_n \beta_1(n) +  \left(\sum_{j=q_n}^n \beta_0(j)\right) + \phi(h_n) + (h_n+n) \eps(n) =0.
    \end{equation*}
  \end{enumerate}
  Hence the result of Theorem \ref{thm:main1} holds in this setting with the sequence $A_n=C_n(y)$.
\end{theorem}

We make two remarks about Theorem \ref{thm:gibbs}. Firstly, since $\nu$ is unique we may assume that $\nu$ was constructed using the machinery of \cite{boggun}, which allows us to assume that we are in the setting of \cite{boggun} throughout the proof of Theorem \ref{thm:gibbs}.
Secondly, in the following lemma we provide a simple sufficient condition that guarantees that $\nu(C_n(y)) > 0$ for every $n \in \ZZ^+$ without \emph{a priori} knowledge of $\nu$. We defer the proof to Section \ref{sec:gibbs}.
\begin{lemma}\label{lemma:pos_measure}
  Suppose that $b \ge 2$ a.s., that \ref{g1} holds, and that $\psi \in \mathbb{F}_A$ satisfies \ref{g2}. Let $\nu$ be the unique $S_Q$-invariant random Gibbs probability measure for $\psi$.
  If $\mathbb{P}\{ \omega : y \in \E_\omega\} > 0$ then $\nu(C_n(y)) > 0$ for every $n \in \ZZ^+$.
\end{lemma}

%%%%%%%%%%%%%%%%%%%%%%%%%%%
\subsection{Proofs for random subshifts}\label{sec:shift}

In this section we will prove Theorem~\ref{thm:main1} by applying Theorem \ref{thm:kallenberg}.
We fix some finite union of disjoint, bounded, open intervals $R=\cup_i R_i$ and apply the computations from Section \ref{sec:gal} with the sets $A=A_n$. We set $k_n=k_*(\omega,R,A_n)$ and $T_n(\omega)=T_\omega^{k_n}(A_n)$ for brevity.
For each $n \in \ZZ^+$ we fix a later-to-be-determined $h_n \ge 0$ and set $G_n(\omega), H_n(\omega), K_n(\omega)$ to be the quantities in Section~\ref{sec:gal} defined with gap $g = h_n+n$.
As per the discussion in Section~\ref{sec:gal} since $\eps(n) \to 0$ we have \ref{k1}, and to deduce \ref{k2} it suffices to prove that $G_n + H_n + K_n \to 0$ for $\mathbb{P}$-a.e. $\omega$.

\begin{lemma} \label{lem:GNas}
	For a.e. $\omega$ we have
	\begin{equation*}
		G_n(\omega) \le (\sup R) \left(h_n \beta_1(n) + \sum_{j=q_n}^n \beta_0(j)\right).
	\end{equation*}
	\begin{proof}
		We have
		\begin{equation*}
		    \mu_{\theta^i \omega}(A_n\cap\{\tau_{A_n} \le h_n+n\}) \le \sum_{j=q_n}^{h_n+n} \mu_{\theta^i \omega}(A_n\cap \sigma^{-j}A_n).
		\end{equation*}
		These intersections will be estimated differently according to whether $j\le n$ or not. Firstly, if $j \le n$ then by \ref{sft1} we have
		\begin{equation*}
      \mu_{\theta^i \omega}(A_n\cap \sigma^{-j}A_n) \le \mu_{\theta^i \omega}(A_n) \beta_0(j).
    \end{equation*}
		Alternatively, if $n<j\le h_h+n$ then by \ref{sft1} we have
		\begin{equation*}
      \mu_{\theta^i \omega}(A_n\cap \sigma^{-j}A_n) \le \mu_{\theta^i \omega}(A_n)\beta_1(n).
    \end{equation*}
		Combining these two cases we get
		\begin{equation}\label{eq:GNas_1}
		    \mu_{\theta^i \omega}(A_n\cap\{\tau_{A_n} \le h_n+n\})
		    \le \mu_{\theta^i \omega}(A_n) \left(h_n \beta_1(n) + \sum_{j=q_n}^n \beta_0(j)\right).
		\end{equation}
    We obtain the required claim by summing \eqref{eq:GNas_1} over $i = 1, \dots, k_n$, recalling the definition of $T_n(\omega)$, and then noting that $T_n(\omega) \le \sup R$.
		\end{proof}
\end{lemma}

\begin{lemma}\label{lem:H_alpha}
  Suppose that \ref{sft3} holds with uniform mixing and rate $\alpha$. Then $H_n(\omega) = O(k_n^{1/p} \alpha(h_n))$ for $\mathbb{P}$-a.e. $\omega$.
  \begin{proof}
  For $i \in \{1, \dots, k_n\}$ we must bound $\Delta_{\theta^i\omega}(A_n,k_n-i,h_n+n)$.
  If $h_n + n \ge k_n-i$ then we are done, having obtained the trivial upper bound of 0, and so we assume otherwise for the remainder of the proof. For $I \subseteq \ZZ^+$ we set $F_n(I)$ to be $F_\omega^{A_n}(I)$ (from \eqref{eq:f_sets}).
  If $I \subseteq \ZZ^+ \cap (h_n + n , k_n-i]$ then $F_n(I) = \sigma^{-h_n -n}(F_n(I'))$ where $I' = I - h_n -n$.
  Since $F_n(I') \in \mathcal{F}_0^{k_n - i - h_n}$ by \ref{sft3} we find that
  \begin{equation*}\begin{split}
    &\abs{\mu_{\theta^i \omega}(C_n(y) \cap F_n(I)) - \mu_{\theta^i \omega}(A_n)\mu_{\theta^i \omega}(F_n(I))} \\
    & = \abs{\mu_{\theta^i \omega}(A_n \cap \sigma^{-h_n -n}F_n(I')) - \mu_{\theta^i \omega}(A_n)\mu_{\theta^{i +h_n +n} \omega}F_n(I')))} \\
    &\le D(\theta^i\omega)\alpha(h_n).
  \end{split}\end{equation*}
  Hence
  \begin{equation*}
  H_n(\omega)\le \alpha(h_n) \sum_{i=1}^{k_n} D(\theta^i\omega)
    \le \alpha(h_n) \left(\sum_{i=1}^{k_n} D(\theta^i\omega)^p\right)^{1/p}.
  \end{equation*}
  By the ergodic theorem we have $\sum_{i=1}^{k} D(\theta^i\omega)^p = O(k \mathbb E(D^p))$
  a.e. $\omega$ as $k\to\infty$, which yields the required claim.
  \end{proof}
\end{lemma}

\begin{lemma}\label{lem:H_phi}
  Suppose that \ref{sft3} holds with $\phi$-mixing. Then $H_n(\omega) \le (\sup R)\phi(h_n)$ for $\mathbb{P}$-a.e. $\omega$.
  \begin{proof}
  We argue as in the proof of the previous lemma to reduce to the case where $h_n + n < k_n - i$, and then observe that if $I \subseteq \ZZ^+ \cap (h_n + n , k_n-i]$ then by \ref{sft3} one has
  \begin{equation*}
    \abs{\mu_{\theta^i \omega}(A_n \cap F_n(I)) - \mu_{\theta^i \omega}(A_n)\mu_{\theta^i \omega}(F_n(I))}
    \le \mu_{\theta^i \omega}(A_n)\phi(h_n).
  \end{equation*}
  Hence
  \begin{equation*}
    H_n(\omega) \le \phi(h_n) \sum_{i=1}^{k_n} \mu_{\theta^i \omega}(A_n) = \phi(h_n) T_n(\omega).
  \end{equation*}
  By definition we have $T_n(\omega) \le \sup R$, which yields the required claim.
  \end{proof}
\end{lemma}

Lastly, in the present setting Lemma~\ref{lem:KisOK} may be rewritten as follows.
\begin{lemma} \label{lem:KNas}
  We have $K_n(\omega) \le (\sup R)(h_n+n) \eps(n)$.
\end{lemma}

We are now in a position to finish the proof of Theorem \ref{thm:main1}.

\begin{proof}[Proof of Theorem~\ref{thm:main1}]
  Recall from our application of Section \ref{sec:gal} that \ref{k1} holds by Lemma \ref{lem:K1_proof}, and hence to prove that $\mathcal{N}_{A_n}(\omega, \cdot)$ converges to a uniform Poisson point process it suffices to verify \ref{k2}.
  By Lemmas~\ref{lem:GNas}, \ref{lem:H_alpha}, \ref{lem:H_phi} and \ref{lem:KNas} we have
  \begin{equation*}
    G_n+H_n+K_n = O\left(h_n \beta_1(n) + \sum_{j=q_n}^n \beta_0(j) + \min(\alpha(h_n)k_n^{\frac1p} , \phi(h_n)) + (h_n+n)\eps(n)\right).
  \end{equation*}
  Hence $G_n+H_n+K_n\to0$ provided that an appropriate sequence $\{h_n\}_{n \in \ZZ^+}$ exists, but such a sequence is guaranteed to exist by~\eqref{eq:zero}.
  The property \ref{k2} then follows from Lemmas~\ref{lem:delta}, \ref{lem:exp} and \ref{lem:somme}.
  Thus $\mathcal{N}_{A_n}(\omega, \cdot)$ converges to a uniform Poisson point process as claimed.
  We will now prove \eqref{eq:cvinlaw}.
  For each $n \in \ZZ^+$ let $f_n : [0, \infty) \to [0,\infty)$ be defined by
  \begin{equation*}
    d_n(s) = \mu_{\omega}(x \in X : \mathcal{N}_{A_n}(\omega, x)([0, s]) = 0).
  \end{equation*}
  Since $\mathcal{N}_{A_n}(\omega, \cdot)$ converges to a uniform Poisson point process we have for every $s \in [0,\infty)$ that
  \begin{equation}\label{eq:pt_hitting_time_prob_conv}
    \lim_{n \to \infty} d_n(s) = e^{-s}.
  \end{equation}
  As each $d_n$ is decreasing and $s \mapsto e^{-s}$ is continuous and decreasing we may conclude that that the convergence in \eqref{eq:pt_hitting_time_prob_conv} is uniform. The argument is a straightforward modification of a well-known one; we include it because it is elementary, short and for lack of an exact reference.
  Since $s \mapsto e^{-s}$ is continuous and decreasing on $[0, \infty)$, for any $\epsilon > 0$ there exists points $0 \le y_0 < \dots < y_m = \infty$ such that $e^{-y_i} - e^{-y_{i+1}} < \epsilon/2$ for every $i \in \{1, \dots, m\}$.
  Fix $s \in [0, \infty)$ and let $i$ be such that $s \in [y_i, y_{i+1})$. If $d_n(s) > e^{-s}$ then as $d_n$ and $s \mapsto e^{-s}$ are decreasing one has
  \begin{equation*}
    \abs{d_n(s) - e^{-s}} \le d_n(y_i) - e^{-y_{i+1}} \le \abs{d_n(y_i) - e^{-y_i}} + \frac{\epsilon}{2} \le \sup_{1 \le j \le m} \abs{d_n(y_j) - e^{-y_j}} + \frac{\epsilon}{2}.
  \end{equation*}
  One obtains the same bound in the case where $d_n(s) \le e^{-s}$ by a similar argument.
  If $n$ is large enough so that $\abs{d_n(y_j) - e^{-y_j}} < \epsilon/2$ for each $j \in \{1, \dots, m\}$ then it follows that
  \begin{equation*}
    \sup_{s \in [0, \infty)} \abs{d_n(s) - e^{-s}} < \epsilon,
  \end{equation*}
  and so the convergence in \eqref{eq:pt_hitting_time_prob_conv} is uniform.
  Hence
  \begin{equation*}
    \lim_{n \to \infty} \sup_{k \ge 0} \abs{d_n\left(T^k_\omega(A_n)\right) - \left(-\sum_{i=1}^k\mu_{\theta^i\omega} (A_n)\right)} = 0.
  \end{equation*}
  The convergence in \eqref{eq:cvinlaw} then follows from the observation that
  \begin{equation*}\begin{split}
    d_n\left(T^k_\omega(A_n)\right) &= \mu_{\omega}\left(x \in X : \mathcal{N}_{A_n}(\omega, x)([0, T^k_\omega(A_n)]) = 0 \right) \\
    &= \mu_{\omega}\left(x \in X : \tau_{A_n}(x) > k\right).
  \end{split}\end{equation*}
\end{proof}

\subsection{Proofs for random Gibbs measures}\label{sec:gibbs}

In this section we will prove Theorem \ref{thm:gibbs} and Lemma \ref{lemma:pos_measure}, starting with the former. In view of Theorem \ref{thm:main1} it suffices to exhibit \ref{sft1}, \ref{sft2}, and \ref{sft3} in order to prove Theorem \ref{thm:gibbs}.
We begin by recalling a uniform quenched decay of correlations result from \cite{morris2006topics}, which strengthens the conclusions of \cite{boggun} and yields \ref{sft3}.

\begin{proposition}\label{prop:quenched_decay}
  There exists $K > 0$ and $\rho \in [0,1)$ such that for every $n,m,k \in \ZZ^+$, $A \in \mathcal{F}^n_0$, $B \in \mathcal{F}^m_0$ we have $\mathbb{P}$-a.s. that
  \begin{equation*}
    \abs{\mu_{\omega}(A \cap \sigma^{-n-k}(B)) - \mu_{\omega}(A) \mu_{\theta^{n+k}\omega}(B)}  \le \mu_{\omega}(A)K\rho^k.
  \end{equation*}
  Hence \ref{sft3} holds with fibered $\phi$-mixing where $\phi(k) = O(\rho^k)$.
  \begin{proof}
    The proof is the same as \cite[Corollary 2.3.7]{morris2006topics} with \cite[Lemma 2.3.5]{morris2006topics} used in place of \cite[Proposition 2.3.6]{morris2006topics}.
  \end{proof}
\end{proposition}

We will now pursue \ref{sft2}, which will be obtained as a corollary to the following result.
\begin{proposition}\label{prop:exp_decay_cylinder}
  There exists $c,a > 0$ such that for all $n \in \ZZ^+$
  \begin{equation*}
    \esssup_{\omega} \sup_{x \in X} \mu_{\omega}(C_{n}(x)) \le ce^{-an}.
  \end{equation*}
\end{proposition}

To prove Proposition \ref{prop:exp_decay_cylinder} we will use the Gibbs property of $\mu$ and a submultiplicativity property of the functions $F_n$ (from Definition \ref{def:gibbs}).
To this end we introduce real-valued functions $L_n$ defined by
\begin{equation*}
  \forall \omega \in \Omega, a,b \in \E_\omega \quad L_n(\omega, a,b) = \exp\left(\sum_{i=0}^{n-1}\psi(\theta^i\omega, \sigma^i a) - \psi(\theta^i\omega, \sigma^i b)\right).
\end{equation*}
Before beginning the proof of Proposition \ref{prop:exp_decay_cylinder} we introduce some handy notation.
If $x, s \in X$ and $n \in \ZZ^+$ then we denote the sequence $(x_0, \dots, x_{n-1}, s_0, s_1, \dots)$ by $[x]_0^n s$.
For $x \in X$, $\omega \in \Omega$ and $n \in \ZZ^+$ we denote by $I_n(\omega, x)$ the set $\{ s \in \E_{\theta^n \omega} :[x]_0^n s \in \E_\omega \} = \sigma^n(\E_{\omega} \cap C_n(x))$.

\begin{lemma}\label{lemma:cylinder_less_than_1}
  There is some $C > 0$ such that
  \begin{equation*}
    \esssup_{\omega} \sup_{x \in \E_{\omega}} \sup_{s \in I_M(\omega, x)} \left(\sum_{u \in \E_\omega \cap \sigma^{-M}(s)} L_M(\omega, u,[x]_0^Ms)\right)^{-1} < (1+C)^{-1}.
  \end{equation*}
  \begin{proof}
    For $x \in \E_\omega$ and $s \in I_M(\omega, x)$ we have
    \begin{equation*}
      \sum_{u \in \E_\omega \cap \sigma^{-M}(s)} L_M(\omega, u, [x]_0^Ms) = 1 + \sum_{\substack{u \in \E_\omega \cap \sigma^{-M}(s)\\ u \ne x}} \exp\left(\sum_{i=0}^{M-1}\psi(\theta^i\omega, \sigma^i u) - \psi(\theta^i\omega, \sigma^i [x]_0^M s)\right).
    \end{equation*}
    By the aperiodicity assumption \ref{g1}, for every $t \in \{ 1, \dots, b(\omega)\}$ there exists $u^t \in \E_\omega$ such that $u^t_0 = t$ and $\sigma^M(u^t) = s$.
    Since $b(\omega) \ge 2$ a.s. it follows that for a.e. $\omega$ there exists $u \in \E_\omega \cap \sigma^{-M}(s)$ such that $u \ne x$ (because we can demand that $u_0$ and $x_0$ are different).
    Using the fact that $\psi$ satisfies \eqref{eq:holder_obs} we deduce that there exists $C > 0$ such that for a.e. $\omega$ and every $x \in \E_\omega$, $s\in I_M(\omega, x)$, and $u \in  \E_\omega \cap \sigma^{-M}(s)$ one has
    \begin{equation*}
      \exp\left(\sum_{i=0}^{M-1}\psi(\theta^i\omega, \sigma^i u) - \psi(\theta^i\omega, \sigma^i [x]_0^M s)\right) \ge C.
    \end{equation*}
    Hence for a.e. $\omega$ we have
    \begin{equation*}
      \sum_{u \in \E_\omega \cap \sigma^{-M}(s)} L_M(\omega, u, [x]_0^Ms) \ge 1 + C
    \end{equation*}
    uniformly in $u,x$ and $s$. The required claim follows.
  \end{proof}
\end{lemma}

\begin{proof}[{The proof of Proposition \ref{prop:exp_decay_cylinder}}]
  It suffices to prove the claim when $n$ is a multiple of $M$, so let $n = kM$ for some $k \in \ZZ^+$.
  If $x \not\in \E_\omega$ and $\mu_{\omega}(C_{kM}(x)) = 0$ then we are done.
  Otherwise, if $x \not\in \E_\omega$ and $\mu_{\omega}(C_{kM}(x)) > 0$ then there must exist some $x' \in \E_\omega \cap C_{kM}(x)$. Since $\mu_{\omega}(C_{kM}(x)) = \mu_{\omega}(C_{kM}(x'))$, by replacing $x'$ with $x$ we may therefore assume that $x \in \E_\omega$ without loss of generality.
  By the Gibbs property and as $x \in \E_\omega$ we have
  \begin{equation*}
    \mu_{\omega}(C_{kM}(x)) = \intf_{I_{kM}(\omega, x)} F_{kM}(\omega, [x]_0^{kM}s) d\mu_{\theta^{kM}\omega}(s).
  \end{equation*}
  Recall that
  \begin{equation}\label{eq:exp_decay_cylinder_1}
    F_{kM}(\omega, [x]_0^{kM}s) = \left(\sum_{u \in \E_\omega \cap \sigma^{-kM}(s)} L_{kM}(\omega, u,[x]_0^{kM}s) \frac{H(\omega,u)}{H(\omega,[x]_0^{kM}s)}\right)^{-1}.
  \end{equation}
  Since $\ln H$ satisfies \eqref{eq:holder_obs} there exists some $C > 0$ such that
  \begin{equation}\label{eq:exp_decay_cylinder_2}
    F_{kM}(\omega, [x]_0^{kM}s) \le C \left(\sum_{u \in \E_\omega \cap \sigma^{-kM}(s)} L_{kM}(\omega, u,[x]_0^{kM}s) \right)^{-1}
  \end{equation}
  for a.e. $\omega$ and every $x \in \E_\omega$ and $s \in I_{kM}(\omega,x)$.
  In the case where $k = 1$ one obtains an upper bound for the right side of \eqref{eq:exp_decay_cylinder_2} from Lemma \ref{lemma:cylinder_less_than_1}.
  For $k > 1$ we have
  \begin{equation}\label{eq:exp_decay_cylinder_3}
    \sum_{u \in \E_\omega \cap \sigma^{-kM}(s)} L_{kM}(\omega, u, [x]_0^{kM}s) = \sum_{\substack{u_1 \in \sigma^{-(k-1)M}(s) \\ u_1 \in \E_{\theta^M \omega}}} \sum_{\substack{u \in  \sigma^{-M}(u_1) \\ u \in \E_{\omega}}} L_{kM}(\omega, [u]_0^M u_1, [x]_0^{kM}s).
  \end{equation}
  Since in \eqref{eq:exp_decay_cylinder_3} we have
  \begin{equation*}
    L_{kM}(\omega, [u]_0^M u_1, [x]_0^{kM}s) = L_{(k-1)M}(\theta^M\omega, u_1, [x]_M^{(k-1)M}s)L_{M}(\omega, u, [x]_0^{kM}s)
  \end{equation*}
  it follows from Lemma \ref{lemma:cylinder_less_than_1} that
  \begin{equation*}
    \sum_{u \in \E_\omega \cap \sigma^{-kM}(s)} L_{kM}(\omega, u, [x]_0^{kM}s) \ge (1+C)\sum_{u_1 \in \sigma^{-(k-1)M}(s) \cap \E_{\theta^M \omega}} L_{(k-1)M}(\theta^M\omega, u_1, [x]_M^{(k-1)M}s).
  \end{equation*}
  By iterating and applying Lemma \ref{lemma:cylinder_less_than_1} again one finds that
  \begin{equation}\label{eq:exp_decay_cylinder_4}
    \left(\sum_{u \in \E_\omega \cap \sigma^{-kM}(s)} L_{kM}(\omega, u, [x]_0^{kM}s)\right)^{-1} \le ((1+C)^{1/M})^{-kM}.
  \end{equation}
  Applying \eqref{eq:exp_decay_cylinder_2} and \eqref{eq:exp_decay_cylinder_4} to \eqref{eq:exp_decay_cylinder_1} one finds that
  \begin{equation*}
    \esssup_{\omega} \sup_{x \in \E_{\omega}} \mu_{\omega}(C_{kM}(x)) \le C((1+C)^{1/M})^{-kM}
  \end{equation*}
  for every $n = kM$ with $k \in \ZZ^+$. As remarked earlier this is sufficient to deduce the result for all $n$.
\end{proof}

\begin{remark}
  An alternative proof of Proposition \ref{prop:exp_decay_cylinder} is provided in \cite[Lemma 2.1]{rousseau2015hitting} (which is adapted from \cite{galves1997inequalities}) under the alternative assumptions of fiber-wise $\psi$-mixing, with summable $\psi$, and the inequality
  \begin{equation}\label{eq:small_cylinder}
      \esssup_{\omega} \sup_{x \in \E_\omega} \mu_{\omega}(C_1(x)) < 1.
  \end{equation}
\end{remark}

Since $\nu(C_n(y)) = \intf \mu_{\omega}(C_n(y)) d\mathbb{P}$, it follows that $\nu(C_n(y)) > 0$ implies that $\eps(n) > 0$.
Hence one deduces \ref{sft2} from Proposition \ref{prop:exp_decay_cylinder}.

\begin{proposition}\label{prop:cylinder_size_gibbs}
  We have \ref{sft2} with $0 < \eps(n) \le ce^{-an}$ for some $c,a > 0$ and every $n \in \ZZ^+$.
\end{proposition}

We now give the proof of \ref{sft1}. The proof uses the usual Gibbs property of $\mu$ on the pseudo-multiplicativity of the measure of cylinder sets, which we reproduce from \cite{boggun} in a modified form.

\begin{lemma}[{\cite[Proposition 4.7]{boggun}}]\label{lemma:distortion}
  There exists $c > 0$ such that for every $n,m \in \ZZ^+$, a.e. $\omega \in \Omega$ and every $x \in \E_\omega$ one has
  \begin{equation*}
    c^{-1} \mu_{\omega}(C_{n}(x))\mu_{\theta^n\omega}(C_{m}(\sigma^n x)) \le \mu_{\omega}(C_{n+m}(x)) \le c\mu_{\omega}(C_{n}(x))\mu_{\theta^n\omega}(C_{m}(\sigma^n x)).
  \end{equation*}
  \begin{proof}
    The result follows from the proof of \cite[Proposition 4.7]{boggun}: it is shown that there are functions $g : \E_Q \to \RR$ and $\lambda : \Omega \to (0, \infty)$ and a constant $c$ such that if $\Psi = \psi + \ln g - \ln (g \circ S_Q) - \lambda$ then for every $n,m \in \ZZ^+$, a.e. $\omega \in \Omega$ and each $x \in \E_\omega$ one has
    \begin{equation}\label{eq:distortion_1}
      c^{-1} \le \frac{\mu_{\omega}(C_{n+m}(x))}{\exp(\sum_{i=0}^{n+m-1}\Psi(\theta^i \omega, \sigma^i x)) } \le c.
    \end{equation}
    Since the denominator of \eqref{eq:distortion_1} is multiplicative, the claimed inequalities readily follows.
  \end{proof}
\end{lemma}

We obtain \ref{sft1} as a consequence of Proposition \ref{prop:exp_decay_cylinder} and Lemma \ref{lemma:distortion}.
\begin{proposition}\label{prop:small_short_returns}
  There exists $c,a > 0$ such that \ref{sft1} holds with $\beta_1(n) = c\exp(-an) = \beta_2(n)$ i.e. for every $j,n \in \NN$ and $\mathbb{P}$-a.s. one has
  \begin{equation}\label{eq:small_short_returns_0}
    \mu_\omega(C_{n}(y)\cap\sigma^{-j}C_{n}(y)) \le c\mu_\omega(C_{n}(y)) \exp(-a\min\{j,n\}).
  \end{equation}
  \begin{proof}
    First suppose that $j \le n$. If $\mu_\omega(C_{n}(y)\cap\sigma^{-j}C_{n}(y)) = 0$ then we are done. On the other hand, if $\mu_\omega(C_{n}(y)\cap\sigma^{-j}C_{n}(y)) > 0$ then there exists some $y' \in \E_\omega \cap C_{n+j}(y)$.
    Hence by Proposition \ref{prop:exp_decay_cylinder} and Lemma \ref{lemma:distortion} we get
    \begin{equation*}\begin{split}
      \mu_\omega(C_{n}(y)\cap\sigma^{-j}C_{n}(y)) &= \mu_\omega(C_{n+j}(y')) \\
      &\le c\mu_\omega(C_{n}(y'))\mu_{\theta^n\omega}(C_{j}(\sigma^n y')) \\
      &\le c\mu_\omega(C_{n}(y))\mu_{\theta^n\omega}(C_{j}(\sigma^n y)) \le c \mu_\omega(C_{n}(y))\exp(-aj).
    \end{split}\end{equation*}
    Thus \eqref{eq:small_short_returns_0} holds when $j \le n$.
    Now suppose that $j > n$.
    Define an equivalence relation $\sim$ on $C_{n}(y)\cap\sigma^{-j}(C_{n}(y)) \cap \E_\omega$ such that $x \sim x'$ iff $x_i = x_i'$ for every $i \in \{n, \dots, j-1\}$.
    Let $P$ denote the set of equivalence classes of $\sim$, and for each $p \in P$ let $p_x$ denote a fixed, representative element of $p$.
    We have
    \begin{equation*}
      \mu_\omega(C_{n}(y)\cap\sigma^{-j}(C_{n}(y))) = \sum_{ p \in P} \mu_\omega(C_{n}(y)\cap \sigma^{-n}(C_{n-j}(p_x))\cap \sigma^{-j}(C_{n}(y))).
    \end{equation*}
    Since $C_{n}(y)\cap \sigma^{-n}(C_{n-j}(p_x))\cap \sigma^{-j}(C_{n}(y)) = C_{n+j}(p_x)$ and $p_x \in \E_\omega$, by applying Lemma \ref{lemma:distortion} twice we get
    \begin{equation}\label{eq:small_short_returns_1}
        \mu_\omega(C_{n}(y)\cap \sigma^{-n}C_{n-j}(p_x)\cap \sigma^{-j}C_{n}(y)) \le c^2 \mu_\omega(C_{n}(y)) \mu_{\theta^j\omega}(C_{n}(y)) \mu_{\theta^n\omega}(C_{n-j}(p_x))
    \end{equation}
    Since the sets $\{C_{n-j}(p_x)\}_{p \in P}$ are pairwise disjoint, by summing \eqref{eq:small_short_returns_1} over $p$ and applying Proposition \ref{prop:exp_decay_cylinder} one gets
    \begin{equation*}\begin{split}
      \mu_\omega(C_{n}(y)\cap\sigma^{-j}(C_{n}(y))) &\le c^2 \mu_\omega(C_{n}(y)) \mu_{\theta^j\omega}(C_{n}(y)) \sum_{p \in P} \mu_{\theta^n\omega}(C_{n-j}(p_x)) \\
      &\le c^3 \mu_\omega(C_{n}(y)) \exp(-an),
    \end{split}\end{equation*}
    which is \eqref{eq:small_short_returns_0} in the case where $j > n$.
  \end{proof}
\end{proposition}

\begin{proof}[{The proof of Theorem \ref{thm:gibbs}}]
   As discussed at beginning of this section, by Theorem \ref{thm:main1} it suffices to exhibit \ref{sft1}, \ref{sft2}, and \ref{sft3}, since \ref{sft4} follows immediately as a result.
   These hypotheses are verified in Propositions \ref{prop:small_short_returns}, \ref{prop:cylinder_size_gibbs} and \ref{prop:quenched_decay}, respectively.
\end{proof}

We finish this section with the proof of Lemma \ref{lemma:pos_measure}.

\begin{proof}[{The proof of Lemma \ref{lemma:pos_measure}}]
  Since $\mathbb{P}\{ \omega : y \in \E_\omega\} > 0$ there exists a set $\Omega_y \subseteq \Omega$ of non-zero $\mathbb{P}$-measure such that $y \in \E_\omega$ for every $\omega \in \Omega_y$. Notice that for each $n \in \ZZ^+$ one has
  \begin{equation}\label{eq:cylinder_lower_bound_0}
    \nu(C_n(y)) \ge \intf_{\Omega_y}  \mu_{\omega}(C_n(y)) d\mathbb{P}(\omega).
  \end{equation}
  By the Gibbs property we have
  \begin{equation}\label{eq:cylinder_lower_bound_1}
    \mu_\omega(C_n(y)) = \intf \sum_{\substack{z \in C_n(y) \cap \E_\omega \\ \sigma^{n+M}(z) = s}} F_{n+M}(\omega, z) d\mu_{\theta^{n+M}\omega}(s).
  \end{equation}
  By the aperiodicity assumption \ref{g1}, for every $s \in \E_{\theta^{n+M}\omega}$ there exists $z \in \E_\omega$ such that $z \in C_n(y)$ and $\sigma^{n+M}(z) = s$.
  Hence the support of the integrand in \eqref{eq:cylinder_lower_bound_1} is all of $\E_{\theta^{n+M}\omega}$.
  Since for $z \in C_n(y) \cap \sigma^{-(n+M)}(s) \cap \E_{\omega}$ one has
  \begin{equation*}
    F_{n+M}(\omega, z) = \left(\sum_{u \in \E_\omega \cap \sigma^{-(n+M)}(z)} L_{n+M}(\omega, u,z) \frac{H(\omega,u)}{H(\omega,z)}\right)^{-1}.
  \end{equation*}
  As $\ln H$ and $\psi$ satisfy \eqref{eq:holder_obs} for each $n$ there exists $c_n > 0$ such that
  \begin{equation}\label{eq:cylinder_lower_bound_2}
    F_{n+M}(\omega, z) \ge c_n\left(\sum_{u \in \E_\omega \cap \sigma^{-(n+M)}(z)} 1 \right)^{-1}.
  \end{equation}
  If $u \in \E_\omega \cap \sigma^{-(n+M)}(z)$ then for each $i \in \{0, \dots, n + M-1\}$ there are at most $b(\theta^i \omega)$ choices for the value of $u_i$, and so
  \begin{equation}\label{eq:cylinder_lower_bound_3}
    \abs{\E_\omega \cap \sigma^{-(n+M)}(z)} \le \prod_{i=0}^{n+M-1} b(\theta^i \omega).
  \end{equation}
  It follows from \eqref{eq:cylinder_lower_bound_2}, \eqref{eq:cylinder_lower_bound_3} and the fact that $C_n(y) \cap \sigma^{-(n+M)}(s) \cap \E_{\omega}$ is non-empty that for every $s \in \E_{\theta^{n+M} \omega}$ we have
  \begin{equation}
     \sum_{\substack{z \in C_n(y) \\ \sigma^{n+M}(z) = s}} F_{n+M}(\omega, z) \ge c_n \prod_{i=0}^{n+M-1} b(\theta^i \omega)^{-1}.
  \end{equation}
  Hence by \eqref{eq:cylinder_lower_bound_0} and \eqref{eq:cylinder_lower_bound_1} we have the lower bound
  \begin{equation*}
    \nu(C_n(y)) \ge c_n \intf_{\Omega_y} \prod_{i=0}^{n+M-1} b(\theta^i \omega)^{-1} d\mathbb{P}(\omega).
  \end{equation*}
  Since $\ln b \in L^1 (\mathbb{P})$ we must have $\intf_{\Omega_y} \prod_{i=0}^{n+M-1} b(\theta^i \omega)^{-1} d\mathbb{P}(\omega) > 0$ for every $n$, and so the required claim follows.
\end{proof}

\section{Examples}\label{sec:examples}

Let $(\Omega, \mathcal{A}, \mathbb{P})$ denote $\RR / \ZZ$ equipped with Lebesgue measure $\Leb$. We identity $\RR / \ZZ$ with $[0,1)$ in the usual way. For $r \in (0,1)$ let $\theta_r$ denote the irrational rotation $x \mapsto x + r$.
Recall that $\theta_r$ is $\Leb$-ergodic but not $\Leb$-mixing.
We will construct a family of random subshifts of finite type that cover $\theta_r$ and to which we may apply Theorem \ref{thm:gibbs}.
%For simplicity's sake we assume that the number of symbols $b$ is  a.s. finite.
The main technical requirement is to exhibit measurable transition matrices $\{ Q(\omega) : \RR^{b(\theta_r \omega)} \to \RR^{b(\omega)} \}_{\omega \in [0,1)}$ that satisfies the uniform aperiodicity condition \ref{g2}. The following proposition provides some simple conditions that guarantee \ref{g2}.

\begin{proposition}\label{prop:aperiodicity}
  Let $r \in (0,1) \setminus \mathbb{Q}$ and suppose that $\{ Q(\omega) : \RR^{b(\theta_r \omega)} \to \RR^{b(\omega)} \}_{\omega \in [0,1)}$ is a random transition matrix over $\theta_r$.
  If there exists an open interval $I \subseteq [0,1)$ such that every entry in $Q(\omega)$ is 1 for every $\omega \in I$ then \ref{g2} is satisfied.
  \begin{proof}
    Since $r$ is irrational, $\theta_r$ is minimal and so there exists some $M$ so that $[0,1) = \bigcup_{k=0}^M \theta_r^{-k}(I)$;
    we claim that \ref{g2} is satisfied with this $M$.
    In particular, for any $\omega \in [0,1)$ there is at least one $k \in \{0, \dots, M-1 \}$ such that $\theta_r^k \omega \in I$. Hence at least one matrix in the product $Q(\omega) \cdots Q(\theta_r^{M-1}\omega)$ has no 0 entries.
    Since each matrix $Q(\theta_r^k \omega)$, $k \in \{0, \dots, M-1 \}$, has at least one 1 in every row and column, it follows that the product $Q(\omega) \cdots Q(\theta_r^{M-1}\omega)$ has no zero entries, which is \ref{g2}.
  \end{proof}
\end{proposition}

\begin{remark}
  It is evident that Proposition \ref{prop:aperiodicity} generalises to the case where $\theta$ is a minimal homeomorphism on a compact metric space and $I$ is any open set.
\end{remark}

\begin{example}
  Let $I = [0,3/4)$ and $J= [0, 1/4) \cup [1/2,1)$. Set
  \begin{equation*}
    Q(\omega) =
    \begin{pmatrix}
      1 & \chi_I(\omega) & \chi_I(\omega)\chi_J(\omega) \\
      \chi_I(\omega) & 1 & \chi_J(\omega) \\
      \chi_I(\omega)\chi_J(\omega) & \chi_J(\omega) & 1 \\
    \end{pmatrix}.
  \end{equation*}
  If $r \in (0,1) \setminus \mathbb{Q}$ then Proposition \ref{prop:aperiodicity} applies to the random transition matrix $Q$, and so we deduce \ref{g2}.
  Let $S_Q : \E_Q \to \E_Q$ denote the random subshift of finite type induced by $Q$ and fix some $\psi \in \mathbb{F}_Q$ satisfying \ref{g1} (notice that since $b= 3$ a.s. it suffices to have $\esssup_{\omega} \sup_{y \in \E_\omega} \psi(\omega, y) < \infty$).
  Denote by $\nu_\psi$ the unique $S_A$-invariant random Gibbs probability measure for $\psi$ that is produced by Theorem \ref{thm:gibbs_existence}, and let $\{\mu_\omega\}_{\omega \in [0,1)}$ denotes the disintegration of $\nu_\psi$.
  Thus Theorem \ref{thm:gibbs} applies to the cylinders $\{C_n(y)\}_{n \in \ZZ^+}$ about any non-periodic $y \in X$ satisfying $\nu_\psi(C_n(y)) > 0$ for each $n \in \ZZ^+$: for $\Leb$ a.e. $\omega$ we have convergence of
  \begin{equation*}
    \N_{C_n(y)}(\omega,x) := \sum_{k=1}^\infty  1_{C_n(y)}(\sigma^k(x)) \delta_{T_\omega^k(C_n(y))}
  \end{equation*}
  to a uniform Poisson Point Process on $[0,\infty)$ under the law of $\mu_\omega$.
  In addition, we have $\Leb$-a.s. convergence in distribution of the first hitting time:
  \begin{equation*}
	   \lim_{n\to\infty}\sup_{k\ge0}\abs{\mu_\omega\left(x\in X\colon R_n(x,y)>k\right)- \exp\left(-\sum_{i=1}^k\mu_{\theta_r^i\omega} (C_n(y))\right)}=0.
  \end{equation*}

   Actually, in this setting the requirement that $\nu_\psi(C_n(y)) > 0$ for every $n \in \ZZ^+$ follows from the existence of some $\omega' \in [0,1)$ such that $y \in \E_\omega$.
   To see why, notice that if $y \in \E_\omega$ then for any $n \in \ZZ^+$ there exists $\omega_n > \omega'$ such that for any $z \in [\omega', \omega_n)$ and $k \in \{0, \dots, n\}$ the points $\theta_r^k(z)$ and $\theta_r^k(\omega')$ are in the same element of the partition $\{[0,1/4), [1/4,1/2), [1/2,3/4), [3/4,1)\}$.
   Hence $Q, \cdots, Q \circ \theta_r^n$ are constant on $[\omega', \omega_n)$ and so by the definition of $\E_Q$ we have $C_n(y) \cap \E_{z} \ne \emptyset$ for each $z \in [\omega', \omega_n)$.
   Since $b = 3$ a.s. one may argue as in the proof of Lemma \ref{lemma:pos_measure} to deduce that $\essinf_{\omega \in \Omega} \inf_{x \in \E_\omega} \mu_{\omega}(C_n(x)) > 0$ for every $n \in \ZZ^+$. Hence for each $n \in \ZZ^+$ we have
   \begin{equation}\begin{split}\label{eq:example_pos}
     \nu_\psi(C_n(y)) &\ge \intf_{[\omega', \omega_n)} \mu_\omega(C_n(y)) d\Leb(\omega) \\
     &> (\omega_n - \omega') \left(\essinf_{\omega \in \Omega} \inf_{x \in \E_\omega} \mu_{\omega}(C_n(x))\right) > 0.
   \end{split}\end{equation}

   Lastly we comment on the fact that while the marginal measure $\mu = \intf \mu_\omega d\Leb(\omega)$ is invariant under the full-shift on 3 symbols, which we will also denote by $\sigma$, $\mu$ is not mixing for $\sigma$.
   Firstly we point out that $S_Q$ possesses a non-trivial symmetry.
   Specifically, for $(\omega,x) \in \E_Q$ set $\omega' = \omega + 1/2$ and let $x'$ denote the sequence obtained by swapping every $1$ with a $3$ in $x$, and vice versa. If the map $(\omega, x) \mapsto (\omega', x')$ is denoted by $\tau$ then one may verify that $\tau$ is a bijective involution (i.e. $\tau = \tau^{-1}$), and that $\tau$ and $S_Q$ commute.
   It follows that $\nu_\psi \circ \tau$ is a $S_Q$-invariant probability measure.
   In addition, if the map $x \mapsto x'$ is denoted by $U : X \to X$ then one may verify that the disintegration of $\nu_\psi \circ \tau$ is given by $\{ \mu_{\omega+1/2} \circ U^{-1} \}_{\omega \in [0,1)}$.
   Since $\nu_\psi$ is a random $S_Q$-invariant Gibbs measure for $\psi$, by checking the definition (Definition \ref{def:gibbs}) it is clear that $\nu_\psi \circ \tau$ is a random $S_Q$-invariant Gibbs measure for the potential $\psi$ too, and so by uniqueness (Theorem \ref{thm:gibbs_existence}) we must have that $\nu_\psi = \nu_\psi \circ \tau$.
   It follows that for every measurable $A \subseteq \E_\omega$ we have
   \begin{equation*}
     \mu_{\omega}(A) = \mu_{\omega+1/2}(U(A)),
   \end{equation*}
   and so for each $n \in \ZZ^+$ and $(\omega,x) \in \E_A$ we have
   \begin{equation}\label{eq:identity}
     \mu_{\omega}(C_n(x)) = \mu_{\omega+1/2}(C_n(x')).
   \end{equation}
   Now fix $\omega' \in [1/4,1/2)$ and notice that by the definition of $Q(\omega')$ there exists some $x \in \E_{\omega'}$ with $x_0 = 0$ and $x_1 = 1$.
   Moreover $C_2(x) \cap \E_{\omega}$ is non-empty when $\omega \in I$, and empty when $\omega \in I^c$.
   Hence by the arguments used to prove \eqref{eq:example_pos} we deduce that the support of $\omega \mapsto \mu_{\omega}(C_2(x))$ is $I$, and as $x \in \E_\omega$ that for every $n \in \ZZ^+$ with $n \ge 2$ the support of $\omega \mapsto \mu_{\omega}(C_n(x))$ has non-zero measure.
   Now let $(\omega',x') = \tau(\omega,x)$ and take $k \ge n$.
   We have
   \begin{equation}\label{eq:marginal_doc}
     \mu(C_n(x) \cap \sigma^{-k} C_n(x')) = \intf \mu_{\omega}(C_n(x)) \mu_{\theta^k_r \omega}(C_n(x')) d\Leb(\omega) + O(\rho^{k-n}).
   \end{equation}
   If $\mu$ was mixing for $\sigma$, then \eqref{eq:marginal_doc} would imply that
   \begin{equation}\label{eq:marginal_doc_2}
     \lim_{k \to \infty} \intf \mu_{\omega}(C_n(x)) \mu_{\theta^k_r \omega}(C_n(x')) d\Leb(\omega) = \mu(C_n(x))\mu(C_n(x')).
   \end{equation}
   We will show that \eqref{eq:marginal_doc_2} is false, which implies that $\mu$ is not mixing.
   Firstly, by \eqref{eq:identity} we have that $\mu(C_n(x)) = \mu(C_n(x'))$.
   Secondly, also by \eqref{eq:identity} we have that
   \begin{equation*}
     \intf \mu_{\omega}(C_n(x)) \mu_{\theta^k_r \omega}(C_n(x')) d\Leb(\omega) = \intf \mu_{\omega}(C_n(x)) \mu_{\omega + 1/2+ kr}(C_n(x)) d\Leb(\omega).
   \end{equation*}
   Since $\theta_r$ is minimal we may find an increasing sequence $\{k_i\}_{i\in \ZZ^+} \subseteq \ZZ^+$ such that $k_ir\to 1/2$.
   Since translation is $L^2$-continuous we have
   \begin{equation*}
     \lim_{i \to \infty} \intf \mu_{\omega}(C_n(x)) \mu_{ \omega+k_ir + 1/2}(C_n(x)) d\Leb(\omega) = \intf \mu_{\omega}(C_n(x))^2 d\Leb(\omega).
   \end{equation*}
   But by Jensen's inequality and as $\omega \mapsto \mu_{\omega}(C_n(x))$ is not a.s. constant (it does not have full support) we have
   \begin{equation*}
     \intf \mu_{\omega}(C_n(x))^2 d\Leb(\omega) > \mu(C_n(x))^2.
   \end{equation*}
   Hence
   \begin{equation*}
     \limsup_{k \to \infty} \intf \mu_{\omega}(C_n(x)) \mu_{\theta^k_r \omega}(C_n(x')) d\Leb(\omega) > \mu(C_n(x))\mu(C_n(x')),
   \end{equation*}
   and therefore $\mu$ is not mixing.
\end{example}

\section*{Acknowledgements}

H.C. is supported by an Australian Government Research Training Program Scholarship and the School of Mathematics and Statistics, UNSW. He would like to thank UNSW for providing travel funding under the Postgraduate Research Student Support (PRSS) Scheme, which allowed him to visit Beno{\^i}t at the Universit{\'e} de Bretagne Occidentale.

%%%%%%%%%%%%%%%%%%%%%%%%

\end{document}